\documentclass[12pt, a4paper]{article}

\usepackage{a4}
\usepackage{latexsym}
\usepackage{color}
\usepackage{amssymb}
\usepackage{amsfonts}
\usepackage{amsthm}
\usepackage{amsmath}
\usepackage[pdftex]{graphicx}
\usepackage{comment}

\usepackage[a4paper,margin=2cm,vmargin={1cm,2cm},includeheadfoot]{geometry}

\newcommand{\eps}{{\varepsilon}}
\newcommand{\E}{{\mathbb{E}}}
\newcommand{\prob}{{\mathbb{P}}}

\newcommand{\convDistrL}{\stackrel{d}{\longrightarrow}}
\newcommand{\eqDistr}{\stackrel{d}{=}}

\newcommand{\convProbL}{\stackrel{\mathbb{P}}{\longrightarrow}}
\newcommand{\convAs}{\stackrel{\mathrm{a.s.}}{\rightarrow}}
\newcommand{\convAsL}{\stackrel{\mathrm{a.s.}}{\longrightarrow}}

\newcommand{\ie}{i.e. }

\newtheorem{theorem}{Theorem}[section]

\newtheorem{proposition}[theorem]{Proposition}
\newtheorem{lemma}[theorem]{Lemma}

\newtheorem{remark}[theorem]{Remark}

\begin{document}
\title{Applying the Wiener-Hopf Monte Carlo simulation technique for L\'evy processes to path functionals}
\author{Albert Ferreiro-Castilla\footnote{
Direcci\'o D'Inversions en Accions,
Banc Sabadell,
Carrer del Sena, 12,
Sant Cugat del Vall\`{e}s 08174, Spain. E-mail: aferreiro.c@gmail.com (Partially supported by a Royal Society Newton International Fellowship).} \quad and \quad Kees van Schaik\footnote{Corresponding Author. School of Mathematics, University of Manchester, Oxford Road, Manchester M13 9PL, United Kingdom. E-mail: kees.vanschaik@manchester.ac.uk}}
\date{\today}
\maketitle

\begin{abstract}
In this note we apply the recently established Wiener-Hopf Monte Carlo (WHMC) simulation technique for L\'evy processes from Kuznetsov et al. \cite{Kuznetsov10c} to path functionals, in particular first passage times, overshoots, undershoots and the last maximum before the passage time. Such functionals have many applications, for instance in finance (the pricing of exotic options in a L\'evy model) and insurance (ruin time, debt at ruin and related quantities for a L\'evy insurance risk process). The technique works for any L\'evy process whose running infimum and supremum evaluated at an independent exponential time allow sampling from. This includes classic examples such as stable processes, subclasses of spectrally one sided L\'evy processes and large new families such as meromorphic L\'evy processes. 
Finally we present some examples. A particular aspect that is illustrated is that the WHMC simulation technique (provided it applies) performs much better at approximating first passage times than a `plain' Monte Carlo simulation technique based on sampling increments of the L\'evy process.
\end{abstract}





\section{Introduction}\label{sec_intro}
\setcounter{equation}{0}

Let $X:=(X_t)_{t \geq 0}$ be a L\'evy process, i.e. a (real valued) stochastic process starting from $0$ with cadlag paths (right continuous and with left limits) and stationary, independent increments whose law we denote by $\mathbb{P}$. A L\'evy process may be thought of as a Brownian motion with drift to which an (infinite) sequence of independent compound Poisson processes are added; infinite to the extent that its small jumps may not be summable. The L\'evy-Khintchine formula entails that the characteristic exponent $\Psi$, defined as $\mathbb{E}[e^{\mathrm{i}zX_t}] = e^{-t \Psi(z)}$ for all $t \geq 0$ and $z \in \mathbb{R}$, can be expressed as
\begin{equation}\label{characExpo} 
\Psi(z) = \frac{\sigma^2}{2} z^2 + \mathrm{i}az + \int_{\mathbb{R} \setminus \{ 0 \}} (1-e^{\mathrm{i}zx} + \mathbf{1}_{\{ |x|<1 \}} \mathrm{i}zx) \Pi(\mathrm{d}x)\ , 
\end{equation}
where $\sigma, a \in \mathbb{R}$ and $\Pi$ a measure on $\mathbb{R} \setminus \{ 0 \}$ satisfying $\int_{\mathbb{R} \setminus \{ 0 \}} (x^2 \wedge 1) \Pi(\mathrm{d}x)<\infty$. See e.g. the textbooks Bertoin \cite{Bertoin96}, Kyprianou \cite{Kyprianou06} or Sato\cite{Sato99} for a detailed introduction. 

Of interest in several fields are quantities of the form
\begin{equation}\label{main1} 
\mathbb{E} \left[ f(\tau_u,X_{\tau_u}-u,u-X_{\tau_u-},u-\overline{X}_{\tau_u-}) \right],
\end{equation}
where $\tau_u$ is the first passage time of $X$ over a level $u>0$, i.e.
\[ \tau_u := \inf \{ t>0 \, | \, X_t>u \}\ , \]
$X_{\tau_u}-u$ is referred to as the overshoot, $u-X_{\tau_u-}$ the undershoot and finally $u-\overline{X}_{\tau_u-}$ the last maximum before first passage. Here and throughout we employ the usual notation $\overline{X}_t := \sup_{s \leq t} X_s$ and $\underline{X}_t := \inf_{s \leq t} X_s$ for all $t \geq 0$. 

Some examples of applications of (\ref{main1}) are as follows. In mathematical finance models driven by L\'evy processes are popular extensions of the classic Black \& Scholes model used for pricing financial products (see e.g. Cont and Tankov \cite{Cont04}, Schoutens \cite{Schoutens03} or Schoutens and Cariboni \cite{Schoutens09}). In such models it is assumed that the financial index on which the payoff of the option is based evolves as $S_t=\exp(X_t)$ for a suitably chosen L\'evy process $X$. The `fair' price of a perpetual American option in such a model is typically determined by the joint law of the first hitting time and the overshoot. Furthermore so-called barrier options are popular tools in practice. The `vanilla' version grants the holder a payoff $g(X_T)$ at a future time $T$ provided $X$ has not or has, dependent on the variation of the product, crossed some barrier $B$ in the meantime. Hence the payoff is a function of $(X_T,\overline{X}_T)$. The original WHMC simulation technique deals with this pair, see also further below. A large variety of more `exotic' versions of such options are popular as well. For instance a certain rebate may be paid at the moment the crossing of the barrier happens. Discrete barrier options exist where the barrier crossing event is only observed at certain subperiods of $[0,T]$, and Parisian barrier options where the barrier condition kicks in only once $X$ has spent at least a given period of time on the `wrong' side of the barrier.

Furthermore in actuarial science a so-called L\'evy insurance risk process is a popular extension of the classic Cram\'er-Lundberg model (cf. Lundberg \cite{Lundberg03}) to study the evolution of the cumulative net of premiums minus claims generated by a homogenuous portfolio of insurance products. See e.g. Asmussen \cite{Asmussen00}, Kluppelberg et al. \cite{Kluppelberg04} or Song and Vondra{\v{c}}ek \cite{Song08}. If the L\'evy insurance risk process is $-X$ and the initial capital $u$ then $\tau_u$ corresponds to the time ruin occurs, i.e. the first time the cumulative net becomes negative. Furthermore $-(X_{\tau_u}-u)$ corresponds to the debt at ruin, $X_{\tau_u}-u$ and $u-X_{\tau_u-}$ give information about the nature of the ruin event --- whether the direct cause is a single large claim or rather the accumulation of many small claims --- and finally $u-\overline{X}_{\tau_u-}$ gives information about how close ruin has been before the actual event. 

These examples illustrate several usecases of (\ref{main1}). To the best of our knowledge, evaluating a quantity like (\ref{main1}) can currently be done using one of two other approaches. We will discuss these alternatives and how they compare with the WHMC simulation method below. 

The first one is a `plain' Monte Carlo simulation method, that is, simulate paths of a random walk whose increments have the same law as $X_h$ for some small $h>0$ as an approximation of the paths of $X$. However, there are only few examples of L\'evy processes $X$ for which the law of $X_h$ is known, in other cases it has to be approximated, typically by a numerical Fourier inversion. This introduces numerical inaccuracy and an extra potentially expensive computation step. Another downside is the well known problem that the empirical law of the simulated first passage time suffers from a very significant bias which vanishes only very slowly as $h$ vanishes, due to the fact that the random walk approach misses excursions over the level $u$ between the grid points. See Broadie et al. \cite{Broadie99} and the references therein. The WHMC simulation method does not suffer from this bias, due to the fact that the method simulates the path of the pair $(X,\overline{X})$ rather than just $X$. See also Subsection \ref{subsec_BM}. Even though this bias is --- to our best knowledge --- only documented for Brownian motion, since any L\'evy process can be decomposed as an independent sum where one of the components is a Brownian motion, there is no reason to expect this issue to be any less significant for a more general L\'evy process or the rest of the quantities involved in (\ref{main1}). 

It should be mentioned that some prominent examples of L\'evy processes used in finance do not allow to apply our method directly. For example, Merton's jump-diffusion model (see Merton \cite{Merton76}) where the driving L\'evy process $X$ is a drifted Brownian motion plus a compound Poisson process with normally distributed jumps. As the Wiener-Hopf factors are not explicitly known in this case, the WHMC cannot be applied while plain Monte Carlo does. The same applies to the popular CGMY (see Carr et al \cite{Carr02}) and NIG (see Barndorff-Nielsen \cite{Barndorff-Nielsen98}) models; however due to the unavailability of the exact law of $X_h$ for a given $h>0$ additional simulation techniques are necessary (see e.g. Chen et al \cite{Chen12} and Glasserman and Liu \cite{Glasserman10}) and an extra error is incorporated to the plain Monte Carlo method. It is worth to mention with respect to the latest examples that a parametrisation of the $\beta$-family (a particular subclass of L\'evy processes for which our methodology is straightforward, see Subsection \ref{beta_family}) is such that the CGMY and NIG processes can be obtained as a limit of $\beta$-processes (see Section 4 in Kuznetsov \cite{Kuznetsov09}). For such cases a dedicated study would be helpful to decide whether or not the advantages the WHMC simulation method has over plain Monte Carlo as described in the previous paragraph outweighs the disadavantage that the WHMC method only applies to an approximation of the actual driving L\'evy process (cf. Ferreiro-Castilla and Schoutens \cite{FS11} or Schoutens and van Damme \cite{SD10} for some results in this direction).

The second approach concerns using the quintuple law from Doney and Kyprianou \cite{Doney06}, which can be written in the form
\begin{multline}\label{GS} 
\mathbb{E} \left[ e^{-q\tau_u} \mathbf{1}_{\{ X_{\tau_u}-u \in \mathrm{d}x, \, u-X_{\tau_u-} \in \mathrm{d}y, \, u-\overline{X}_{\tau_u-} \in \mathrm{d}z \}} \right] \\
= \frac{1}{q} \mathbb{P} \left( \overline{X}_{\mathrm{e}(q)} \in u-\mathrm{d}z \right) \mathbb{P} \left( -\underline{X}_{\mathrm{e}(q)} \in \mathrm{d}y-z \right) \Pi(\mathrm{d}x+y) 
\end{multline}
for $q,x,y>0$ and $z \in [0,u \vee y]$. Here and throughout $\mathrm{e}(q)$ denotes an exponentially distributed random variable with mean $1/q$, independent of $X$. Hence, if we know the laws of $\overline{X}_{\mathrm{e}(q)}$ and $\underline{X}_{\mathrm{e}(q)}$ --- which is exactly the condition under which our simulation method can be implemented --- we might use this result to compute (\ref{main1}). However, to obtain the law of $\tau_u$ from the above expression we would need to invert the right hand side over $q$ which is in general not a very straightforward operation, see for instance Section \ref{sec_examples} for how the laws of $\overline{X}_{\mathrm{e}(q)}$ and $\underline{X}_{\mathrm{e}(q)}$ depend on $q$ when $X$ is a meromorphic L\'evy process. Furthermore, obtaining (\ref{main1}) still requires computing a potentially four dimensional integral, for which in general a numerical method will be required. Since the simulation method we propose in this note is a straightforward, easy to implement and efficient method for approximating (\ref{main1}) directly, with only the laws of $\overline{X}_{\mathrm{e}(q)}$ and $\underline{X}_{\mathrm{e}(q)}$ as input, in general there seems no reason not to prefer it over the quintuple law alternative.

To return to the focus of this note, in Kuznetsov et al. \cite{Kuznetsov10c} the Wiener-Hopf Monte Carlo (WHMC) simulating technique is introduced which allows to sample from a law that is a good approximation of the law of $(X_T,\overline{X}_T)$, provided that samples can be produced from $\overline{X}_{\mathrm{e}(q)}$ and $\underline{X}_{\mathrm{e}(q)}$. This method was extended to a multilevel version and a theoretical analysis was given in Ferreiro-Castilla et al. \cite{FKSS12}. In this note we pursue the observation that the main idea behind the WHMC simulating technique can also be used to generate samples from (an approximation of)

\begin{equation}\label{tuple}
(\tau_u,X_{\tau_u}-u,u-X_{\tau_u-},u-\overline{X}_{\tau_u-})
\end{equation}
rather than only (an approximation of) $(X_T,\overline{X}_T)$. In fact, not only (\ref{tuple}) but any functional of the pair $(X,\overline{X})$ could be handled by the method, cf. Remark \ref{rem1}. Once this observation has been established it is simply a matter of applying the usual setup: generate a large number of such samples, apply the function $f$ to each of them and compute the resulting average to obtain an approximation of (\ref{main1}).

For simplicity, in the sequel we will refer to (\ref{tuple}) as \emph{the $4$-tuple}. 

The rest of this note is organized as follows. In Section \ref{WH_scheme} we recall the original WHMC simulating technique and discuss how the underlying idea is useful for path functionals as well. Section \ref{sec_proofs} is devoted to the main result, describing how to obtain an approximation of (\ref{tuple}) in terms of $\overline{X}_{\mathrm{e}(q)}$ and $\underline{X}_{\mathrm{e}(q)}$, and the convergence in distribution of this approximation to the exact law of (\ref{tuple}). In Section \ref{first_passage_sec} we explore the convergence rate of the approximation of the first passage time, as this is the key quantity involved in (\ref{main1}). We will show that \mbox{$\E[(\widehat{\tau}_{u}^{n}-\tau_{u})^{2}]=\mathcal{O}(n^{-1})$}, where $\widehat{\tau}_{u}^{n}$ is the approximation of $\tau_{u}$ as given by the extended WHMC simulating technique. We will also show that the extended WHMC simulating technique admits a multilevel version to approximate $\tau_{u}$ and that such enhancement makes the algorithm optimal. Finally, the implementation, some examples and some numerical results supporting the theoretical claims and exposing the practical side of the extended WHMC simulating technique are collected in Section \ref{sec_examples}. 

\section{The WHMC simulating technique}\label{WH_scheme}
\setcounter{equation}{0}

Let us shortly recall the WHMC simulating technique as introduced in Kuznetsov et al. \cite{Kuznetsov10c} and discuss how that setup can be used for path functionals as well. Fix some $t>0$. The idea in Kuznetsov et al. \cite{Kuznetsov10c} is to construct an approximation of the joint law of $(X_t,\overline{X}_t)$ making use of a 'stochastic grid' on the time axis together with the Wiener-Hopf factorisation as follows. Recall that $\mathrm{e}(\lambda)$ denotes an exponentially distributed random variable with mean $1/\lambda$ independent of $X$. For any $n \geq 1$, enlarge the probability space on which $X$ lives with an i.i.d. sequence $\{\mathrm{e}_i(n/t)\}_{i \geq 1}$ and define a set of grid points as
\begin{equation}\label{def_grid} 
g(0,n/t):=0, \quad g(k,n/t) := \sum_{i=1}^k \mathrm{e}_i(n/t) \quad \mbox{for $k \geq 1$}. 
\end{equation}
For any $n$ the set of random points $\{ 0=g(0,n/t)<g(1,n/t)<... \}$ forms a grid on the time axis, the distance between the grid points forming a sequence of i.i.d. exponentially distributed random variables; equivalently the grid points can be seen as the arrival times of a Poisson process with rate $n/t$. For convenience we still denote the law of $X$ and the grid by $\mathbb{P}$. The idea of using such a 'stochastic grid' was first coined in Carr \cite{Carr98} in the context of finding the value of American put options in a Black \& Scholes model. 

The usefulness of this setup to sample from an approximation of $(X_t,\overline{X}_t)$ relies on the following three facts. 

Firstly, the celebrated Wiener-Hopf factorisation 
tells us that for any L\'evy process $X$ and $q>0$ we have
\begin{equation}\label{WH} 
(X_{\mathrm{e}(q)},\overline{X}_{\mathrm{e}(q)}) \eqDistr (\overline{X}_{\mathrm{e}(q)} + \underline{X}_{\mathrm{e}(q)},\overline{X}_{\mathrm{e}(q)}). 
\end{equation}

Secondly, making use of stationary independent increments of $X$, it can be shown that the above equality can be extended in the sense that we have for any $n$
\[ (X_{g(n,n/t)},\overline{X}_{g(n,n/t)}) \eqDistr (V_n^{(n/t)},J_n^{(n/t)}) \]
where $V_n^{(n/t)}$ and $J_n^{(n/t)}$ are random variables on a new probability space whose law can be expressed in terms of the laws of  $\overline{X}_{\mathrm{e}(n/t)}$ and $\underline{X}_{\mathrm{e}(n/t)}$ in a straightforward way. See Theorem \ref{thm_sims_Kuzn} below.
\begin{theorem}[{Kuznetsov et al. \cite[Theorem 1]{Kuznetsov10c}}]\label{thm_sims_Kuzn} Suppose $\lambda>0$. Let $S_{0}^{(\lambda)}=I_{0}^{(\lambda)}:=0$ and let $(S_{i}^{(\lambda)})_{i \geq 1}$ (resp. $(I_{i}^{(\lambda)})_{i \geq 1}$) be a sequence of i.i.d. random variables with common law equal to the law of $\overline{X}_{\mathrm{e}(\lambda)}$ (resp. $\underline{X}_{\mathrm{e}(\lambda)}$). Then we have for any $n \in \mathbb{N}$
\[ (X_{g(n,\lambda)},\overline{X}_{g(n,\lambda)}) \eqDistr (V_n^{(\lambda)},J_n^{(\lambda)}) \]
where $V_n^{(\lambda)}$ and $J_n^{(\lambda)}$ are iteratively defined by $V_0^{(\lambda)}=J_0^{(\lambda)}=0$ and for $i \geq 1$:
\[ V_i^{(\lambda)} = V_{i-1}^{(\lambda)} + S_{i}^{(\lambda)} + I_{i}^{(\lambda)} \quad \mbox{and} \quad J_i^{(\lambda)} =\max \left\{ J_{i-1}^{(\lambda)}, V_{i-1}^{(\lambda)} + S_{i}^{(\lambda)} \right\}. \]
\end{theorem}
\noindent Since $g(n,n/t) \convAs t$ as $n \to \infty$ by the law of large numbers and since $X$ does not jump at fixed times, we have $(X_{g(n,n/t)},\overline{X}_{g(n,n/t)}) \convAs (X_t,\overline{X}_t)$ as $n \to \infty$ and hence for large $n$ the law of $(V_n^{(n/t)},J_n^{(n/t)})$ provides an approximation of the joint law of $(X_t,\overline{X}_t)$. 

Thirdly and finally, especially in recent years many families of L\'evy processes have emerged for which the laws of both $\overline{X}_{\mathrm{e}(n/t)}$ and $\underline{X}_{\mathrm{e}(n/t)}$ are known in explicit enough form to allow sampling from. Besides classical examples such as stable processes, we can mention the class of spectrally one sided L\'evy processes for which the so-called scale functions are explicit enough (see e.g. Hubalek and Kyprianou \cite{Hubalek10}, Kyprianou \cite{Kyprianou06} or Kyprianou et al. \cite{Kyprianou10}), meromorphic L\'evy processes (cf. Kuznetsov et al. \cite{Kuznetsov10}) and Vigon's \cite{Vigon02} technique for constructing new L\'evy processes from prespecified ladder height processes (and hence, through the Wiener-Hopf factorization, with prespecified laws for $\overline{X}_{\mathrm{e}(n/t)}$ and $\underline{X}_{\mathrm{e}(n/t)}$).

Now let us discuss how the above setup is useful for producing samples from the 4-tuple as well. The idea is quite straightforward: the 'stochastic grid' as defined in (\ref{def_grid}) does not only satisfy $g(n,n/t) \convAs t$ as $n \to \infty$, but for any sequence $k(n)$ such that $k(n) \in \{0,\ldots,n\}$ and $k(n)t/n \convAs s \in [0,t]$ as $n \to \infty$, we have again by the law of large numbers $g(k(n),n/t) \convAs s$. Consequently, as above, for large $n$ the law of $(V_{k(n)}^{(n/t)},J_{k(n)}^{(n/t)})$ provides an approximation of the law of $(X_s,\overline{X}_s)$. In this sense the 'stochastic grid' becomes dense in the interval $[0,t]$. Furthermore, as is obvious from Theorem \ref{thm_sims_Kuzn}, due to the iterative nature of the definitions of $V$ and $J$, obtaining a sample from the pair $(V_{n}^{(n/t)},J_{n}^{(n/t)})$ requires producing a sample from the vector $((V_{0}^{(n/t)},J_{0}^{(n/t)}),\ldots,(V_{n}^{(n/t)},J_{n}^{(n/t)}))$. Hence constructing the approximative law of $(X_t,\overline{X}_t)$ automatically yields an approximative law of the vector $((X_0,\overline{X}_0),(X_{1/n},\overline{X}_{1/n}),\ldots,(X_{t},\overline{X}_{t}))$ --- see Proposition \ref{random_walk_approx} below --- and it is therefore at least intuitively clear that we should also be able to approximate a quantity like (\ref{main1}). This is made rigorous in Theorem \ref{thm_sims}. The difficulty to be overcome is that convergence on the 'stochastic grid' is less obvious than on a traditional deterministic grid.

\begin{proposition}\label{random_walk_approx}
Let $X$ be a L\'evy process, $\lambda>0$ and recall $V$ and $J$ as defined in Theorem \ref{thm_sims_Kuzn} and the definition of the stochastic grid in (\ref{def_grid}). Then
\begin{equation}\label{4dec1} 
\left( (X_{g(0,\lambda)},\overline{X}_{g(0,\lambda)}), \ldots, (X_{g(k,\lambda)},\overline{X}_{g(k,\lambda)}) \right) \eqDistr \left(  (V_0^{(\lambda)},J_0^{(\lambda)}), \ldots, (V_k^{(\lambda)},J_k^{(\lambda)}) \right)\ . 
\end{equation}
\end{proposition}

\begin{proof} This is a straightforward adaptation of the proof of Theorem 1 in Kuznetsov et al.  \cite{Kuznetsov10c}, we include it here for completeness.
The proof is by induction over $k$. It is trivially true for $k=1$ on account of (\ref{WH}). Let $k \geq 2$. We have, where $Y$ is an independent copy of $X$ and we use the notation $\overline{X}_{s_1,s_2} := \sup_{s_1 \leq u \leq s_2} X_u$:
\begin{align*}
& \left( (X_{g(0,\lambda)},\overline{X}_{g(0,\lambda)}), \ldots, (X_{g(k,\lambda)},\overline{X}_{g(k,\lambda)}) \right) \\ 
&= \left( (X_{g(0,\lambda)},\overline{X}_{g(0,\lambda)}), \ldots, (X_{g(k-1,\lambda)},\overline{X}_{g(k-1,\lambda)}), \left( X_{g(k,\lambda)},\max \left\{ \overline{X}_{0,g(k-1,\lambda)}, \overline{X}_{g(k-1,\lambda),g(k,\lambda)} \right\} \right) \right) \\
&\eqDistr \bigl( ( X_{g(0,\lambda)},\overline{X}_{g(0,\lambda)}), \ldots, (X_{g(k-1,\lambda)},\overline{X}_{g(k-1,\lambda)}), \\
& \quad \quad \quad \quad \quad \quad \left( X_{g(k-1,\lambda)}+Y_{\mathrm{e}(\lambda)},\max \left\{ \overline{X}_{0,g(k-1,\lambda)}, X_{g(k-1,\lambda)}+\overline{Y}_{\mathrm{e}(\lambda)} \right\} \right) \bigr) \\
&\eqDistr \left( (V_0^{(\lambda)},J_0^{(\lambda)}), \ldots, (V_{k-1}^{(\lambda)},J_{k-1}^{(\lambda)}), \left( V_{k-1}^{(\lambda)}+S_k^{(\lambda)}+I_k^{(\lambda)}, \max \left\{ J_{k-1}^{(\lambda)}, V_{k-1}^{(\lambda)}+S_k^{(\lambda)} \right\} \right) \right) \\
&= \left( (V_0^{(\lambda)},J_0^{(\lambda)}), \ldots, (V_{k-1}^{(\lambda)},J_{k-1}^{(\lambda)}), (V_{k}^{(\lambda)},J_{k}^{(\lambda)}) \right),
\end{align*}
where the second equality uses that $X$ has stationary, independent increments and the third equality uses (\ref{WH}) together with the induction hypothesis and the definition of the sequences $(S_i^{(\lambda)})_{i \geq 1}$ and $(I_i^{(\lambda)})_{i \geq 1}$.
\end{proof}

There is yet another heuristic justification to support the skeleton $\{X_{g(k,n/t)}\}_{k\geq0}$ 
as a good random walk approximation of the L\'evy process to compute pathwise quantities. It can be inferred from Doney \cite{Doney04} that, for all $k>0$, the random variables
\begin{equation*}
M_{k}:=\sup_{g(k,n/t)\leq t<g(k+1,n/t)}X_{t}
\quad
\text{ and }
\quad
m_{k}:=\inf_{g(k,n/t)\leq t<g(k+1,n/t)}X_{t}
\end{equation*}
can be written as
\begin{equation*}
M_{k}=S_{0}^{(n/t)}+Y_{k}^{(+)}
\quad
\text{ and }
\quad
m_{k}=I_{0}^{(n/t)}+Y_{k}^{(-)}\ ,
\end{equation*}
where $\{Y_{k}^{(+)}\}_{k\geq0}$ and $\{Y_{k}^{(-)}\}_{k\geq0}$ are random walks with the same distribution as $\{X_{g(k,n/t)}\}_{k\geq0}$ and independent of $S_{0}^{(n/t)}$ and $I_{0}^{(n/t)}$ respectively. Since it is clear that
\begin{equation*}
m_{k}\leq X_{t}\leq M_{k}\quad \text{ for }\quad g(k,n/t)\leq t<g(k+1,n/t)\ ,
\end{equation*}
the derivations in Doney \cite{Doney04} assert that it is possible to `stochastically' bound the path of $X$ from above and below by two random walks which are equal in distribution to the skeleton constructed in Proposition 
\ref{random_walk_approx}, but with different random starting points. Heuristically, 
the random walk produced in Proposition \ref{random_walk_approx} should be particularly useful when used to approximate pathwise quantities of $X$.

\section{Approximate distribution of the $4$-tuple}\label{sec_proofs}
\setcounter{equation}{0}

Let us now show how the setup introduced in the above Section \ref{WH_scheme} can be used to generate an approximate distribution of the 4-tuple as well. The idea is to approximate $\tau_u$ by finding points on our 'stochastic grid' (\ref{def_grid}) enclosing it, i.e. $k(n) \in \mathbb{N}$ such that $g(k(n)-1,n/t) \leq \tau_u \leq g(k(n),n/t)$ for all $n \in \mathbb{N}$, and evaluate the functionals involving overshoots and undershoots using these grid points.

\begin{theorem}\label{thm_sims} Let $X$ be any L\'evy process. Fix some $t>0$ and $u>0$. Recall $V$ and $J$ as defined in Theorem \ref{thm_sims_Kuzn}. Set for all $n \in \mathbb{N}$
\[ \kappa^{(n)}_u := \inf \{ k \in \{0,\ldots,n\} \, | \, J_k^{(n/t)}>u \} \]
(where as usual we understand $\inf \emptyset=\infty$). Then we have as $n \to \infty$
\begin{multline}\label{eq_conv_d}
\left( \frac{t}{n} (\kappa^{(n)}_u \wedge n), V_{\kappa^{(n)}_u \wedge n}^{(n/t)}-u, u-V_{(\kappa^{(n)}_u-1) \wedge n}^{(n/t)}, u-J_{(\kappa^{(n)}_u-1) \wedge n}^{(n/t)} \right)  \\
\convDistrL \left( \tau_u \wedge t, X_{\tau_u \wedge t}-u, u-X_{(\tau_u \wedge t)-}, u-\overline{X}_{(\tau_u \wedge t)-} \right)\ .  
\end{multline}
\end{theorem}

Before we prove the above main result let us prove two technical lemmas, the first of which is a well know result reproduced here for the sake of completeness and the second one is a computation which will appear in the proof of Theorem \ref{thm_sims}.

\begin{lemma}\label{lem_help} Suppose $Z$ is a compound Poisson process with only positive jumps following an exponential distribution with mean $1/\theta$. Then for any $u>0$ and $\eps \in [0,u)$ 
\[ \mathbb{P} ( Z_{\tau_u}-u > \eps) = \mathbb{P} ( u-Z_{\tau_u-} > \eps ) = e^{-\theta \eps}. \]
\end{lemma}

\begin{proof} Condition on the event that first passage over $u$ happens at the $n$-th jump, this jump is now from an independent random level below $u$ over $u$ and follows an exponential distribution. Hence, using the lack of memory property the result for the overshoot follows. For the undershoot $u-Z_{\tau_u-}$, one can condition on the same event and use that the undershoot is the difference between $u$ and the sum of $n-1$ independent exponentially distributed random variables to conclude the proof. 
\end{proof}

\begin{lemma}\label{tech_lemma1}
Let $X$ be a L\'evy process, $T$ a random time, $(A^{(n)})_{n \geq 1}$ a sequence of events and $(T^{(n)})_{n \geq 1}$ a sequence of random times such that
\[ \mathbf{1}_{A^{(n)}} (T^{(n)}-T) \convProbL 0 \ \mbox{ as $n \to \infty$.} \]
\begin{itemize}
\item[(i)] If for all $n$ we have $T^{(n)} \geq T$ on $A^{(n)}$ a.s., then 
\begin{equation}\label{8dec2}
\mathbf{1}_{A^{(n)}} (X_{T^{(n)}} - X_{T}) \convProbL 0\ \mbox{ as $n \to \infty$.}
\end{equation}
\item[(ii)] If for all $n$ we have $T^{(n)} < T$ on $A^{(n)}$ a.s., then
\begin{equation}\label{8dec3}
\mathbf{1}_{A^{(n)}} (X_{T^{(n)}} - X_{T-}) \convProbL 0\ \mbox{ and }\ \mathbf{1}_{A^{(n)}} (\overline{X}_{T^{(n)}} - \overline{X}_{T-}) \convProbL 0\ \mbox{ as $n \to \infty$.}
\end{equation}
\end{itemize}
\end{lemma}
\begin{proof}
For (\ref{8dec2}), fix some $\eps>0$ and let $\eps'>0$, then we have
\begin{align*}
\mathbb{P}( \mathbf{1}_{A^{(n)}} | X_{T^{(n)}}-X_{T} |>\eps ) &= \mathbb{P}( \mathbf{1}_{A^{(n)}} (\mathbf{1}_{\{ T^{(n)} \in [T,T+\eps'] \}} + \mathbf{1}_{\{ T^{(n)} >T+\eps'] \}}) | X_{T^{(n)}}-X_{T} |>\eps ) \\
&\leq \mathbb{P}\left( \sup_{s \in [0,\eps']} |X_{T+s}-X_{T}| >\eps \right) + \mathbb{P}( \mathbf{1}_{A^{(n)}} (T^{(n)} >T+\eps') ),
\end{align*}
where the second term vanishes as $n \to \infty$ by assumption, while the first can be made arbitrarily small by choosing $\eps'$ small enough on account of right continuity and path regularity. Clearly (\ref{8dec3}) follows by similar means and left continuity of $(X_{s-})_{s \geq 0}$.
\end{proof}

\begin{proof}[Proof of Theorem \ref{thm_sims}] 
The proof goes by defining an auxiliary random vector which would be equal in distribution to the left hand side of (\ref{eq_conv_d}) and which is such that converges in probability to the right hand side of (\ref{eq_conv_d}), thus obtaining the claim of the statement. In order to define such auxiliary random vector let us first introduce the following quantities:
\begin{align*}
k_{X}^{(n)} &:= \inf \{ k \in \{0,\ldots,n \} \, | \, \overline{X}_{g(k,n/t)} >u \}\ ,\\
k_{g}^{(n)}&:=\inf \{ g(k,n/t) \, | \, g(k,n/t) > \tau_u \}\ ,\\
\sigma_+^{(n)} &:= \mathbf{1}_{\{ k_{X}^{(n)}<\infty \}} g(k_{X}^{(n)},n/t) + \mathbf{1}_{\{ k_{X}^{(n)}=\infty \}} g(n,n/t)\ , \\
\sigma_-^{(n)} &:= \mathbf{1}_{\{ k_{X}^{(n)}<\infty \}} g(k_{X}^{(n)}-1,n/t) + \mathbf{1}_{\{ k_{X}^{(n)}=\infty \}} g(n,n/t)\ .
\end{align*}
We first observe some relationships between the above variables. Note that we may write
\begin{equation}\label{8dec6}
\mathbf{1}_{\{ k_{X}^{(n)}<\infty \}} \sigma_+^{(n)} = \mathbf{1}_{\{ k_{X}^{(n)}<\infty \}} k_{g}^{(n)}\ .
\end{equation}
Indeed, on the event $\{ \overline{X}_{g(k,n/t)} \leq u \}$ we have on the one hand $\tau_u \geq g(k,n/t)$ and hence $k_{g}^{(n)}>g(k,n/t)$; on the other hand we have $k_{X}^{(n)}>k$ and hence $\sigma_+^{(n)}>g(k,n/t)$. On the event $\{ \overline{X}_{g(k,n/t)} > u \}$, since $X$ does not jump at fixed times it neither does at $g(k,n/t)$ and thus $\tau_u<g(k,n/t)$, so $k_{g}^{(n)} \leq g(k,n/t)$; moreover, $k_{X}^{(n)} \leq k$ and hence $\sigma_+^{(n)} \leq g(k,n/t)$. Summarizing up, on the event
$\{ \overline{X}_{g(k,n/t)} \leq u \}$ we have that $k_{g}^{(n)},\sigma_+^{(n)}>g(k,n/t)$ and on the event $\{ \overline{X}_{g(k,n/t)} > u \}$ we have $k_{g}^{(n)},\sigma_+^{(n)}\leq g(k,n/t)$ which proves (\ref{8dec6}) since $k_{g}^{(n)}$ and $\sigma_+^{(n)}$ can only take values on the stochastic grid. Furthermore, we have
\begin{equation}\label{8dec7}
\mathbf{1}_{\{{k}_{X}^{(n)}<\infty\}}\sigma_-^{(n)}<\mathbf{1}_{\{{k}_{X}^{(n)}<\infty\}}\tau_u\ ,
\end{equation}
since by construction $\sigma_-^{(n)} \leq \tau_u$ on the event $\{ {k}_{X}^{(n)}<\infty \}$, and $\sigma_-^{(n)} = \tau_u$ can not happen a.s.  due to the fact that $X$ and $\{g(k,n/t)\}_{k=0}^{n}$ are independent. 

Now, with these definitions we have
\begin{multline*} 
\left( \frac{t}{n} ({k}_{X}^{(n)} \wedge n), X_{\sigma_+^{(n)}}-u, u-X_{\sigma_-^{(n)}}, u-\overline{X}_{\sigma_-^{(n)}} \right) \\ 
\eqDistr \left( \frac{t}{n} (\kappa^{(n)}_u \wedge n), V_{\kappa^{(n)}_u \wedge n}^{(n/t)}-u, u-V_{(\kappa^{(n)}_u-1) \wedge n}^{(n/t)}, u-J_{(\kappa^{(n)}_u-1) \wedge n}^{(n/t)} \right), 
\end{multline*}
which follows from applying the same functional to the left and right hand side of (\ref{4dec1}) (with $\lambda=n/t$). Hence the main result follows if we show
\begin{multline}\label{main_th_aim}
\left( \frac{t}{n} ({k}_{X}^{(n)} \wedge n), X_{\sigma_+^{(n)}}-u, u-X_{\sigma_-^{(n)}}, u-\overline{X}_{\sigma_-^{(n)}} \right) \\ 
\convProbL \left( \tau_u \wedge t, X_{\tau_u \wedge t}-u, u-X_{(\tau_u \wedge t)-}, u-\overline{X}_{(\tau_u \wedge t)-} \right) \quad \mbox{as $n \to \infty$}\ , 
\end{multline}
for which in turn it is enough to show that each component on the left hand side converges in probability to its counterpart on the right hand side. This takes up the remainder of the proof.

To prove the convergence in probability of the first component in (\ref{main_th_aim}), note that since $g(n,n/t) \convAs t$ as $n \to \infty$ we have
\begin{equation*}
\mathbf{1}_{\{ {k}_{X}^{(n)}=\infty\}}\left(t-\tau_{u}\wedge t\right)
\leq \mathbf{1}_{\{ {k}_{X}^{(n)}=\infty \, , \, \tau_{u}< t\}}t
=\mathbf{1}_{\{ g(n,n/t)<\tau_{u} \, , \, \tau_{u}< t\}}t\convAsL 0\quad \mbox{as $n \to \infty$}\ .
\end{equation*}
Similarly 
\begin{equation*}
\mathbf{1}_{\{ {k}_{X}^{(n)}<\infty \, , \, \tau_u \geq t \}} \left| \frac{t}{n} {k}_{X}^{(n)}-t  \right| 
\leq \mathbf{1}_{\{ g(n,n/t)>\tau_{u} \, , \, \tau_u > t \}}t
\convAsL 0 \quad \mbox{as $n \to \infty$}\ , 
\end{equation*}
and hence it is enough to show that
\begin{equation}\label{first_component}
\mathbf{1}_{\{ {k}_{X}^{(n)}<\infty \, , \, \tau_u < t \}} \left( \frac{t}{n} {k}_{X}^{(n)} - \tau_u \right) \convProbL 0 \quad \mbox{as $n \to \infty$}\ . 
\end{equation}
To prove (\ref{first_component}), let $\eps>0$ and take $0=t_0<t_1<\ldots<t_N=t$ such that $t_i-t_{i-1}<\eps/2$ for all $i$. Then
\begin{multline}\label{first_component_ineq} 
\mathbb{P} \left( \mathbf{1}_{\{ {k}_{X}^{(n)}<\infty \, , \, \tau_u < t \}} \left| \frac{t}{n} {k}_{X}^{(n)} - \tau_u \right| >\eps \right) = \sum_{i=1}^N \mathbb{P} \left( \mathbf{1}_{\{ {k}_{X}^{(n)}<\infty \, , \, \tau_u \in [t_{i-1},t_i) \}} \left| \frac{t}{n} {k}_{X}^{(n)} - \tau_u \right| >\eps \right) \\
\leq \sum_{i=1}^N \mathbb{P} \left( t_{i-1}-\frac{t}{n} \chi^{(n)}(t_{i-1}) >\frac{\eps}{2} \right) + \mathbb{P} \left( \frac{t}{n} \chi^{(n)}(t_{i})-t_{i} > \frac{\eps}{2} \right),
\end{multline}
where for any $x \geq 0$
\[ \chi^{(n)}(x) := \inf \{ k \geq 0 \, | \, g(k,n/t)>x \}\ . \]
Fix $1\leq i\leq N$ and let $\overline{k}(n)=n(t_i+\eps/2)/t$ and $\underline{k}(n)=n(t_{i-1}-\eps/2)/t$ for $n\geq1$, then we see 
\begin{align}
\mathbb{P} \left( \frac{t}{n} \chi^{(n)}(t_{i})-t_{i} > \frac{\eps}{2} \right) 
&\leq \mathbb{P} (\chi^{(n)}(t_i) > \overline{k}(n)) = \mathbb{P} (g(\overline{k}(n),n/t)\leq t_i)\rightarrow 0 \label{aux_firs_component1}\\
\mathbb{P} \left( t_{i-1}-\frac{t}{n} \chi^{(n)}(t_{i-1}) >\frac{\eps}{2} \right)
&\leq \mathbb{P} ( \underline{k}(n)>\chi^{(n)}(t_{i-1}) ) = \mathbb{P} (t_{i-1}\leq g(\underline{k}(n),n/t))\rightarrow 0 \label{aux_firs_component2}
\end{align}
as $n\to\infty$ since the law of large numbers ensures that 
\begin{equation*}
g(\overline{k}(n),n/t)\convAsL t_i+\eps/2>t_i
\quad\text{ and }\quad
g(\underline{k}(n),n/t)\convAsL t_{i-1}-\eps/2<t_{i-1}\ .
\end{equation*}
Finally, we use (\ref{aux_firs_component1}) and (\ref{aux_firs_component2}) in (\ref{first_component_ineq}) to conclude (\ref{first_component}).

For the second component of (\ref{main_th_aim}), we use again the convergence of $g(n,n/t) \to t$ and the fact that $X$ does not jump at fixed times to note
\[ \mathbf{1}_{\{ {k}_{X}^{(n)}=\infty \}} X_{\sigma_+^{(n)}} = \mathbf{1}_{\{ \overline{X}_{g(n,n/t)} \leq u \}} X_{g(n,n/t)} \convAsL \mathbf{1}_{\{ \overline{X}_{t} \leq u \}} X_{t} = \mathbf{1}_{\{ \tau_u \geq t \}} X_t. \]
Using this with $\mathbf{1}_{\{ {k}_{X}^{(n)}<\infty \}} \convAs \mathbf{1}_{\{ \tau_u < t \}}$ it remains to show that
\[ \mathbf{1}_{\{ {k}_{X}^{(n)}<\infty \, , \, \tau_u < t \}} (X_{\sigma_+^{(n)}} - X_{\tau_u}) \convProbL 0 \quad \mbox{as $n \to \infty$}. \]
In virtue of (\ref{8dec2}) in Lemma \ref{tech_lemma1} and recalling (\ref{8dec6}) to check that 
$\sigma_+^{(n)}=k^{(n)}_{g}>\tau_{u}$ on $\{{k}_{X}^{(n)}<\infty\}$, the above limit will hold if we show that
\begin{equation*}
\mathbf{1}_{\{ {k}_{X}^{(n)}<\infty \, , \, \tau_u < t \}} (\sigma_+^{(n)}-\tau_u)
\convProbL 0 \quad \mbox{as $n \to \infty$}.
\end{equation*}
For any $\eps>0$ recall the partition of $[0,t]$ in (\ref{first_component_ineq}). Then
\begin{align}
\mathbb{P} \left( \mathbf{1}_{\{ {k}_{X}^{(n)}<\infty \, , \, \tau_u < t \}} (\sigma_+^{(n)}-\tau_u)>\eps \right) &= \sum_{i=1}^N \mathbb{P} \left( \mathbf{1}_{\{ {k}_{X}^{(n)}<\infty \, , \, \tau_u  \in [t_{i-1},t_i) \}} (\sigma_+^{(n)}-\tau_u)>\eps \right) \notag\\
 &\leq \sum_{i=1}^N \mathbb{P} \left( \upsilon^{(n)}(t_{i}) -t_{i-1}>\eps \right) \notag\\
 &= \sum_{i=1}^N \mathbb{P} \left( \upsilon^{(n)}(t_{i}) -t_{i}>\frac{\eps}{2} \right)\label{aux_second_comp}
\end{align}
where for any $x \geq 0$
\[ \upsilon^{(n)}(x) := \inf \{ g(k,n/t) \geq 0 \, | \, g(k,n/t)>x \}\ . \]
Since $\upsilon^{(n)}$ is nothing but the first passage time over $x$ by a compound Poisson process with exponential jumps with mean $t/n$ we see from Lemma \ref{lem_help} that the right hand side of (\ref{aux_second_comp}) indeed vanishes as $n \to \infty$.

For the third and forth component of (\ref{main_th_aim}), a similar argument as the one in the beginning of the previous paragraph will tell us that it is enough to show
\[ \mathbf{1}_{\{ {k}_{X}^{(n)}<\infty \, , \, \tau_u < t \}} (X_{\sigma_-^{(n)}} - X_{\tau_u-}) \convProbL 0 
\ \text{ and }\ 
\mathbf{1}_{\{ {k}_{X}^{(n)}<\infty \, , \, \tau_u < t \}} (\overline{X}_{\sigma_-^{(n)}} - \overline{X}_{\tau_u-}) \convProbL 0
\]
as $n \to \infty$. Recalling (\ref{8dec7}) to check the assumptions of (\ref{8dec3}) in Lemma \ref{tech_lemma1}, both limits above will hold if we prove 
\[ \mathbf{1}_{\{ {k}_{X}^{(n)}<\infty \, , \, \tau_u < t \}} (\sigma_-^{(n)}-\tau_u) \convProbL 0 \quad \mbox{as $n \to \infty$}. \]
For this we may apply the obvious analogue to the argument in (\ref{aux_second_comp}), with the understanding that rather than $\upsilon^{(n)}(x)-x$ we now need to make use of $$x-\sup \{ g(k,n/t)\geq0 \, | \, g(k,n/t)<x \}\ ,$$ as it is apparent from the definition of $\sigma_-^{(n)}$ and (\ref{8dec7}). However this expression is nothing but the undershoot of a compound Poisson process with exponential jumps at first passage over $x$, hence Lemma \ref{lem_help} again applies to yield the result.
\end{proof}

\begin{remark}\label{rem1}
As will be clear from the setup and the above proof, there is nothing special about the entries in the 4-tuple other than that they can be expressed as path functionals of the pair $(X,\overline{X})$. Any other functional of this type can be handled by the method as well, and by symmetry so can path functionals of the pair $(X,\underline{X})$. It is worth noting though that path functionals depending on both $\overline{X}$ and $\underline{X}$ can (in general) not be handled by the current method as there is no analogue of (\ref{WH}) for the pair $(\overline{X}_{\mathrm{e}(q)},\underline{X}_{\mathrm{e}(q)})$.
\end{remark}

\begin{remark}
One might argue that a weakness of the WHMC simulation method is the fact that we replace the fixed time $t$ by a random variable $g(n,n/t)$ and the a priori error this causes. In Ferreiro-Castilla et al. \cite{FKSS12} a comprehensive error analysis is carried out for the bivariate distribution $(X_{g(n,n/t)},\overline{X}_{g(n,n/t)})$ from where convergence rates are derived in terms of the moments of $g(n,n/t)$. For a L\'evy process with finite second moment it is deduced that $\E[(X_{g(n,n/t)}-X_{t})^{2}]=\mathcal{O}(n^{-1/2})$. Reassuringly, the WHMC simulation technique clearly outperforms `plain' Monte Carlo when the running supremum gets involved. Cf. also Subsection \ref{subsec_BM} below.
\end{remark}

\section{Convergence rate of the first passage time}\label{first_passage_sec}
\setcounter{equation}{0}

All quantities involved in (\ref{main1}) ultimately depend on the first passage time.
We now derive a convergence rate for the approximation of the first passage time in order to gain some
insight into the efficiency of the Monte Carlo scheme based on the construction in Theorem 
\ref{thm_sims}. Recall from the proof of Theorem \ref{thm_sims} that $\kappa^{(n)}_u \eqDistr {k}_{X}^{(n)}$ where ${k}_{X}^{(n)}$ lives on the same probability space as $X$.

\begin{theorem}\label{conv_rate_tau}
Using the same notation as in Theorem \ref{thm_sims}, we have
\begin{equation*}
\E\left[\left( \frac{t}{n} ({k}_{X}^{(n)} \wedge n)-\tau\wedge t\right)^{2}\right]\leq \frac{2t^{2}}{n}
\end{equation*}
\end{theorem}
\begin{proof}
Note that an alternative definition of ${k}_{X}^{(n)}$ in the proof of Theorem \ref{thm_sims} is
\begin{equation*}
k_{X}^{(n)} := \inf \{ k \in \{0,\ldots,n \} \, | \, g(k,n/t) >\tau_u \}
\end{equation*}
and hence, conditioned on $\tau_{u}$, $k_{X}^{(n)}$ follows a truncated Poisson distribution. Let us write $\alpha=\tau/t$ to ease the notation and denote by $\mathbf{N}:=\{\mathbf{N}_{t}\}_{t\geq0}$ a Poisson process with rate $n/t$. We then write
\begin{multline}
\frac{1}{t^{2}}\E\left[\left.\left(\frac{t}{n}(k_{X}^{(n)}\wedge n)-\tau_{u}\wedge t\right)^{2}\right|\tau_{u}\right]
=\underbrace{\E\left[\left.\left(1-\alpha\wedge 1\right)^{2}\right|\tau_{u}\right]
\prob\left(\left.k_{X}^{(n)}= \infty\right|\tau_{u}\right)}_{(I)}\\
+\ 
\underbrace{\E\left[\left.\left(\frac{k_{X}^{(n)}}{n}-\alpha\wedge 1\right)^{2}\right|\tau_{u}\right]\prob\left(\left.k_{X}^{(n)}\leq n\right|\tau_{u}\right)}_{(II)}\ .
\label{original}
\end{multline}
We note that when $\tau_{u}\geq t$ ($\alpha\geq1$) the term $(I)$ vanishes and hence
\begin{align}
(I)
&\leq\mathbf{1}_{\{\alpha<1\}}(1-\alpha)^{2}\prob\left(\left.k_{X}^{(n)}=\infty\right|\tau_{u}\right)\notag\\
&=\mathbf{1}_{\{\alpha<1\}}(1-\alpha)^{2}\prob\left(\left.\mathbf{N}_{\tau_{u}}>n\right|\tau_{u}\right)\notag\\
&\leq\mathbf{1}_{\{\alpha<1\}}(1-\alpha)^{2}\frac{\alpha n}{\alpha n+(1-\alpha)^{2}n^{2}}\leq \frac{1}{n}\ ,\label{first_leg}
\end{align}
where the last equality follows from the observation that $\mathbf{N}_{\tau_{u}}$ has a Poisson distribution with parameter $\tau_{u}n/t=\alpha n$ and the one-sided Chebyshev inequality, i.e.
\begin{equation*}
\prob(Z-\mu\geq a)\leq\frac{\sigma^{2}}{\sigma^{2}+a^{2}}
\end{equation*}
for a real random variable $Z$ with mean $\mu$ and variance $\sigma^{2}$.
We now bound the probability factor in $(II)$ by $1$ and note 
\begin{align}
(II)
&\leq
\mathbf{1}_{\{\alpha\geq1\}}
\sum_{i=1}^{n}\left(\frac{i}{n}-1\right)^{2}e^{-\alpha n}\frac{(\alpha n)^{i}}{i!}
+
\mathbf{1}_{\{\alpha<1\}}
\sum_{i=1}^{n}\left(\frac{i}{n}-\alpha\right)^{2}e^{-\alpha n}\frac{(\alpha n)^{i}}{i!}\notag\\
&\leq
\mathbf{1}_{\{\alpha\geq1\}}
e^{(1-\alpha)n}\alpha^{n}
\sum_{i=1}^{n}\left(\frac{i}{n}-1\right)^{2}e^{-n}\frac{n^{i}}{i!}
+
\mathbf{1}_{\{\alpha<1\}}
\sum_{i=1}^{\infty}\left(\frac{i}{n}-\alpha\right)^{2}e^{-\alpha n}\frac{(\alpha n)^{i}}{i!}\notag\\
&\leq
\mathbf{1}_{\{\alpha\geq1\}}
e^{(1-\alpha)n}\alpha^{n}
\frac{1}{n}
+
\mathbf{1}_{\{\alpha<1\}}
\frac{\alpha}{n}\leq \frac{1}{n}\ ,\label{second_leg}
\end{align}
where the last equality follows from $(1+x)\leq e^{x}$ for $x\geq0$.

The claim of the statement follows by noting that (\ref{first_leg}) and (\ref{second_leg}) are upper bounds independent from $\tau_{u}$ and applying the tower property in (\ref{original}).
\end{proof}

The convergence rates for the rest of the tuple seem less straightforward to derive, which is due to the fact that the error for the other entries in the tuple (\ref{tuple}) depend more heavily on the path of $X$. To briefly illustrate that this is a somewhat subtle problem, it seems that if $X$ is any L\'evy process with jumps it holds that 

\begin{equation}\label{Kevin} 
\mathbb{E} \left[ \max_{k \leq n} \left( X_{g(k,n/t)}-X_{kt/n} \right)^2 \right] 
\end{equation}
does not vanish as $n \to \infty$. (This claim is currently supported by numerical evidence only.) Indeed, if a joint realisation of the stochastic grid and of the process $X$ is such that at least one jump of $X$ occurs in the set 

\[ J_n := \bigcup_{k \leq n} (\min\{g(k,n/t),kt/n\},\max\{g(k,n/t),kt/n\}] \]
then the random variable inside the expectation in (\ref{Kevin}) is bounded below by the size of that jump squared. Hence (\ref{Kevin}) could only vanish if in the limit no jumps occur in $J_n$. However some numerical experiments conducted by the authors suggest that the Lebesgue measure of $J_n$ tends to $t$ as $n \to \infty$ which suggests this requirement is not satisfied.

Nevertheless, even though the vanishing of (\ref{Kevin}) would facilitate deriving error bounds for the other entries of the tuple in the spirit of Theorem \ref{conv_rate_tau}, this is not a necessary condition. Indeed, it could very well be the case that (\ref{Kevin}) does not vanish while the expected squared error for the remaining entries in the tuple does vanish. In Section \ref{sec_examples} some numerical results can be found, suggesting this is indeed the case and yielding empirical numerical convergence rates.

\subsection{Multilevel Monte Carlo schemes for the first passage time}

In Ferreiro-Castilla et al. \cite{FKSS12} the original WHMC simulation technique as introduced in Section \ref{WH_scheme}, i.e. to approximate quantities of the form $\E[f(X_{t},\overline{X}_{t})]$, was extended to a multilevel Monte Carlo algorithm under the assumption that $f$ is uniformly Lipschitz. In principle the same multilevel scheme presented in that paper can be used to approximate (\ref{main1}) as well, provided that $X$ has finite second moment (and provided that the random tuple approximation converges in mean square error). It is out of the scope of this paper to give the full details of such a scheme, we refer the reader to Giles \cite{Giles08} for the general theory of multilevel Monte Carlo methods and to Ferreiro-Castilla et al. \cite{FKSS12} for the specifics in the adaptation to the Wiener-Hopf scheme for L\'evy processes. However since we have derived the convergence rate for the first passage time approximation in the above Theorem \ref{conv_rate_tau} we are able to derive the convergence rate for its multilevel version as well. Its numerical performance is checked in Section \ref{sec_examples}. 

For the following discussion it is enough to consider $f:[0,t]\to\mathbb{R}$ Lipschitz with constant $1$, fix $n_{0},L\in\mathbb{N}$ and set $n_{\ell}=2^{\ell}n_{0}$ for $\ell=0,\ldots,L$. Let us write
\begin{equation*}
f^{n_{\ell}}:=f\left(\frac{t}{n_{\ell}} (\kappa^{(n_{\ell})}_u \wedge n_{\ell})\right)\ .
\end{equation*}
The multilevel Monte Carlo algorithm proposes to estimate $\E[f(\tau_{u}\wedge t)]$ by $\E[f^{n_{L}}]$ according to the right hand side of
\begin{equation}\label{MLMC}
\E[f^{n_{L}}]=\E[f^{n_{0}}]+\sum_{\ell=1}^{L}\E[f^{n_{\ell}}-f^{n_{\ell-1}}]\ .
\end{equation}
In other words, the multilevel version of a Monte Carlo estimator proposes to perform regular Monte Carlo estimators in each of the expectations in the right hand side of (\ref{MLMC}) as an approximation to $\E[f(\tau_{u}\wedge t)]$, i.e.
\begin{equation*}
\E[f(\tau_{u}\wedge t)]\approx
\widehat{f}_{\mathrm{ML}}^{\mathcal{M}(n_0,L)}:=
\frac{1}{M_0}\sum_{i=1}^{M_0}f^{n_0,(i)} +
\sum_{\ell=1}^{L} \frac{1}{M_{\ell}}\sum_{i=1}^{M_{\ell}}\left(f^{n_\ell,(i)}-f^{n_{\ell-1},(i)}\right)
\ ,
\end{equation*}
where $\mathcal{M}(n_0,L):=\{M_{\ell}\}_{\ell=0}^L$ are the Monte Carlo trials in each level and $f^{n_\ell,(i)}$ denotes the 
$i$-th sample from $f^{n_\ell}$. The following theorem describes the gain of this approach:
\begin{theorem}[{Ferreiro-Castilla et al. \cite{FKSS12}}]\label{Rob_thm}
Let $t>0$ and $n_\ell = n_0 2^\ell$, for some $\ell,n_0 \in \mathbb{N}$, suppose that there are positive constants $\alpha, \beta > 0$ with
$\alpha\geq\frac{1}{2}(\beta\wedge1)$ such that
\begin{enumerate}
\item[(i)] $|\E[f^{n_\ell}-f(\tau_{u}\wedge t)]|\lesssim n_\ell^{-\alpha}$
\item[(ii)] $\mathbb{V}(f^{n_\ell}-f^{n_{\ell-1}})\lesssim n_\ell^{-\beta}$
\item[(iii)] $\E[\mathcal{C}_{n_\ell}]\lesssim n_\ell$,
\end{enumerate}
where $\mathcal{C}_{n_{\ell}}$ represents the cost of computing a single sample of $f^{n_{\ell}}$. Then, for every $\nu\in \mathbb{R}_{>0}$, there exists a value $L$ and a sequence $\mathcal{M}(n_0,L) = \{M_{\ell}\}_{\ell=0}^L$ such that
\begin{equation*}
\E\left[\mathcal{C}\left(\widehat{f}_{\mathrm{ML}}^{\mathcal{M}(n_0,L)}\right)\right] \lesssim \nu \quad \text{and} \quad 
\E\left[\left(\widehat{f}_{\mathrm{ML}}^{\mathcal{M}(n_0,L)}-f(\tau_{u}\wedge t)\right)^{2}\right] 
\lesssim \left\{
\begin{array}{ll}
\nu^{-\frac{1}{2}}\,,&\text{if }\ \beta>1\,,\\
\nu^{-\frac{1}{2}}\log\nu\,, &\text{if }\ \beta=1\,,\\
\nu^{-\frac{1}{2+(1-\beta)/\alpha}}\,, &\text{if } \ \beta<1\,.
\end{array}
\right.
\end{equation*}
\end{theorem}

In Section \ref{sec_examples} below we perform some numerical simulations for a L\'evy process $X$ belonging to the class of 
so-called meromorphic L\'evy processes (cf. Subsection \ref{beta_family}). It is inferred from Ferreiro-Castilla et al. \cite{FKSS12} that any functional applied to the random walk generated in (\ref{4dec1}) satisfies condition (iii) above. The Lipschitz assumption on $f$ and the triangle inequality together with Theorem \ref{conv_rate_tau} ensure that $\beta=1$ and therefore, up to logarithms, the multilevel Monte Carlo estimator of $f(\tau_{u}\wedge t)$ provided by the above Theorem \ref{Rob_thm} is optimal. Note that in the regime $\beta=1$ the bias does not play a role in the convergence of the algorithm. That is, the scheme proposed here used in 
a multilevel Monte Carlo estimator of the quantity $\E\left[f(\tau_{u}\wedge t)\right]$ 
is optimal.

\section{Numerical implementation}\label{sec_examples}
\setcounter{equation}{0}

In this section we discuss some examples of implementations of the method. Our aim is to show that
the method is easily implemented, has a clear advantage over a `plain' Monte Carlo approach (cf. Subsection \ref{subsec_BM}) and give some intuition about the convergence rates of the method. There are already other studies available which highlight the numerical performance of the WHMC simulation technique in terms of computation time (cf. Ferreiro-Castilla et al. \cite{FKSS12} or Schoutens and van Damme \cite{SD10}), at least for the original setup to approximate quantities depending on $(X_t,\overline{X}_t)$ for some $t>0$ (cf. Section \ref{WH_scheme}). For this reason we focus here on a more qualitative numerical analysis in order to gain further insight into how the WHMC simulation technique performs and we omit speed comparisons which ultimately depend on the particular problem in hands.
 
\subsection{The $\beta$-family}\label{beta_family}

With the exception of Subsection \ref{subsec_BM} we choose $X$ to be a member of the $\beta$-family of L\'evy processes to perform the numerical experiments. This is a subclass, like the $\theta$-processes (cf. Kuznetsov \cite{Kuznetsov10b}) and general hypergeometric L\'evy processes (cf. Kuznetsov et al. \cite{Kuznetsov10c}), of the family of so-called meromorphic L\'evy processes as recently introduced in Kuznetsov et al. \cite{Kuznetsov10}. The class of meromorphic L\'evy processes is very rich: paths of bounded and unbounded variation and both finite and infinite activity jumps can be generated. In terms of the triplet (cf. (\ref{characExpo})) a meromorphic L\'evy process can be characterized by any $\sigma, a$ and any L\'evy measure that can be written as a (possibly infinite) mixture of exponential distributions on $\mathbb{R}_{<0}$ and $\mathbb{R}_{>0}$. The name comes from the fact that their characteristic exponent $\Psi$ can be extended to a meromorphic function on $\mathbb{C}$. 

In particular, a member of the $\beta$-family has a L\'evy measure with a density $\pi$ given by
\begin{equation}\label{beta_density}
\pi(x) = \mathbf{1}_{\{ x>0 \}} c_1 \frac{e^{-\alpha_1 \beta_1 x}}{(1-e^{-\beta_1 x})^{\lambda_1}}
+\mathbf{1}_{\{ x<0 \}} c_2 \frac{e^{\alpha_2 \beta_2 x}}{(1-e^{\beta_2 x})^{\lambda_2}}   
\end{equation}
where $\alpha_i,\beta_i>0$, $c_i \geq 0$ and $\lambda_i \in (0,3)$. For our simulation method we need to know the laws of $\overline{X}_{\mathrm{e}(q)}$ and $\underline{X}_{\mathrm{e}(q)}$; they are easily obtained as follows (cf. Kuznetsov \cite{Kuznetsov09} for more details). With the usual notation $B(x,y)=\Gamma(x)\Gamma(y)/\Gamma(x+y)$ and $\psi(x)=\frac{\mathrm{d}}{\mathrm{d}x} \log \Gamma(x)$, the characteristic exponent (\ref{characExpo}) of the  L\'evy-Khintchine representation can be written as
\[ \Psi(z) = \frac{\sigma^2}{2}z^2 + \mathrm{i} \rho z - \frac{c_1}{\beta_1} B\left( \alpha_1-\frac{\mathrm{i}z}{\beta_1},1-\lambda_1 \right) - \frac{c_2}{\beta_2} B\left( \alpha_2+\frac{\mathrm{i}z}{\beta_2},1-\lambda_2 \right)+\gamma\ , \]
where 
\begin{align*}
\gamma &= \frac{c_1}{\beta_1} B\left( \alpha_1,1-\lambda_1 \right) + \frac{c_2}{\beta_2} B\left( \alpha_2,1-\lambda_2 \right)\ ,\\
\rho &= \frac{c_1}{\beta_1^2} B\left( \alpha_1,1-\lambda_1 \right) (\psi(1+\alpha_1-\lambda_1)-\psi(\alpha_1))\\
&\qquad - \frac{c_2}{\beta_2^2} B\left( \alpha_2,1-\lambda_2 \right) (\psi(1+\alpha_2-\lambda_2)-\psi(\alpha_2)) -a\ .
\end{align*}
For any $q>0$, the equation $\zeta \mapsto q+\Psi(\mathrm{i}\zeta)$ has infinitely many zeros, all real and simple, located as follows:
\begin{align*}
&\zeta_0^- \in (-\beta_1 \alpha_1,0), \, \zeta_0^+ \in (0,\alpha_2 \beta_2),\\
&\zeta_k^- \in (\beta_1 (k-\alpha_1),\beta_1 (k+1-\alpha_1)),\text{ for }  k\leq -1\\
&\zeta_k^+ \in (\beta_2(\alpha_2+k-1),\beta_2(\alpha_2+k)), \text{ for } k\geq1\ .
\end{align*}
Finally, under the above notation the characteristic functions for $\overline{X}_{\mathrm{e}(q)}$ and $\underline{X}_{\mathrm{e}(q)}$ are given by 
\begin{equation}\label{infi_chf}
\varphi_{q}^{+}(z)=\mathbb{E}[e^{iz\overline{X}_{\mathrm{e}(q)}}]
=
\prod_{n\leq0}\frac{1+\frac{iz}{\beta_{1}(n-\alpha_{1})}}{1+\frac{iz}{\zeta_{n}^{-}}}
\ ,\ \ 
\varphi_{q}^{-}(z)=\mathbb{E}[e^{iz\underline{X}_{\mathrm{e}(q)}}]
=
\prod_{n\geq0}\frac{1+\frac{iz}{\beta_{2}(n+\alpha_{2})}}{1+\frac{iz}{\zeta_{n}^{+}}}\ .
\end{equation}
In order to implement the WHMC simulation technique we need to sample from the random variables $\overline{X}_{\mathrm{e}(q)}$ and $\underline{X}_{\mathrm{e}(q)}$. This hence requires truncation of the infinite product representation in (\ref{infi_chf}). (Note that this applies to any meromorphic L\'evy process, since in that case $\varphi_{q}^{+}$ and $\varphi_{q}^{-}$ take the same infinite product form albeit with different roots and poles generated by the function $\zeta \mapsto q+\Psi(\mathrm{i}\zeta)$). For example, if we truncate $\varphi_{q}^{-}$ after $N$ factors then we effectively sample from the random variable $\underline{X}_{\mathrm{e}(q)}^{N}$ with characteristic function
\begin{equation}\label{finite_prod}
\varphi_{q,N}^{-}(z)=\mathbb{E}[e^{iz\underline{X}^{N}_{\mathrm{e}(q)}}]
=
\prod_{n\geq0}^{N}\frac{1+\frac{iz}{\beta_{2}(n+\alpha_{2})}}{1+\frac{iz}{\zeta_{n}^{+}}}\ .
\end{equation}
Samples from $\underline{X}_{\mathrm{e}(q)}^{N}$ are easily obtained using the observation that each factor of the finite product in (\ref{finite_prod}) can be rewritten as
\begin{equation*}
\frac{1+\frac{iz}{\beta_{2}(n+\alpha_{2})}}{1+\frac{iz}{\zeta_{n}^{+}}}
=
\frac{\zeta_{n}^{+}}{\beta_{2}(n+\alpha_{2})}
+\left(1-\frac{\zeta_{n}^{+}}{\beta_{2}(n+\alpha_{2})}\right)\left(1+\frac{iz}{\zeta_{n}^{+}}\right)^{-1}\ .
\end{equation*}
The above is nothing but the characteristic function of a measure consisting of an atom in zero plus a defective negative exponential distribution, \ie
\begin{equation}\label{KK}
\frac{\zeta_{n}^{+}}{\beta_{2}(n+\alpha_{2})}\delta_{0}
+\left(1-\frac{\zeta_{n}^{+}}{\beta_{2}(n+\alpha_{2})}\right)\mathrm{e}(-\zeta_{n}^{+})\ ,
\end{equation}
and it is straightforward to obtain samples according to such a measure. Hence the law of $\underline{X}_{\mathrm{e}(q)}^{N}$ can be expressed as a finite sum of independent random variables with a probability measure as in (\ref{KK}). A similar construction is valid for the supremum.
The above observation provides a very straightforward way to simulate the supremum and the infimum for general meromorphic L\'evy processes and, as far as we know, novel. Furthermore in Ferreiro-Castilla and Schoutens \cite{FS11} the mean squared error due to the truncation of the infinite product was derived:
\begin{equation*}
\E\left[\left(\underline{X}_{\mathrm{e}(q)}-\underline{X}_{\mathrm{e}(q)}^{N}\right)^{2}\right]\leq
\frac{3}{(\zeta_{N+1}^{+})^{2}}\leq\frac{3}{\beta_2^{2}(\alpha_2+N)^{2}}\lesssim
\mathcal{O}(N^{-2})\ .
\end{equation*} 

The recent literature working with meromorphic L\'evy processes typically inverts the distribution functions of $\underline{X}_{\mathrm{e}(q)}$ and $\overline{X}_{\mathrm{e}(q)}$ which are also available (cf. Kuznetsov \cite{Kuznetsov09}) for producing samples. However these distribution functions are expressed as infinite sums, hence truncation has to be applied in this approach as well. Furthermore the truncated distribution function has to be inverted numerically. Since this has to happen very efficiently this inevitably introduces additional error, and it it is not easy to analyse how such error influences the end result. The above approach does not suffer from this problem.

\subsection{First passage time of Brownian motion}\label{subsec_BM}

In this subsection we take $X$ to be a standard Brownian motion. Our goal is to present some numerical evidence for the fact that `plain' Monte Carlo (i.e., a random walk approximation of $X$ by sampling increments of $X$) yields a significant bias when it comes to approximating the first passage time, as already alluded to in the Introduction. As is well known we have $\mathbb{P}(\tau_u \leq s) = 2(1-\Phi(u/\sqrt{2s}))$ for $s>0$, furthermore both $\overline{X}_{\mathrm{e}(q)}$ and $-\underline{X}_{\mathrm{e}(q)}$ follow an exponential distribution with mean $1/\sqrt{2q}$. It is hence straightforward to implement the WHMC simulation technique and approximate $\tau_u$. Since in the WHMC simulation technique a time step requires two samples, namely one from $\overline{X}_{\mathrm{e}(q)}$ and one from $\underline{X}_{\mathrm{e}(q)}$ (cf. Section \ref{sec_proofs}) we take for the plain Monte Carlo half the step size (i.e. $1/(2n)$). 

Figure \ref{Fig1} shows a plot of $t \mapsto \mathbb{P}(\tau_2 \leq t)$ for $t \in[0,50]$. Figure \ref{Fig2} shows the errors of both plain Monte Carlo and the WHMC simulation technique. Note that also if we decrease the step size and increase the number of samples considerably for the plain Monte Carlo approximation, the corresponding error remains several factors larger than the error the WHMC simulation technique makes. 

\begin{figure}[t]
  \begin{minipage}[t]{.45\textwidth}
    \begin{center}  
      \includegraphics[width=7.5cm]{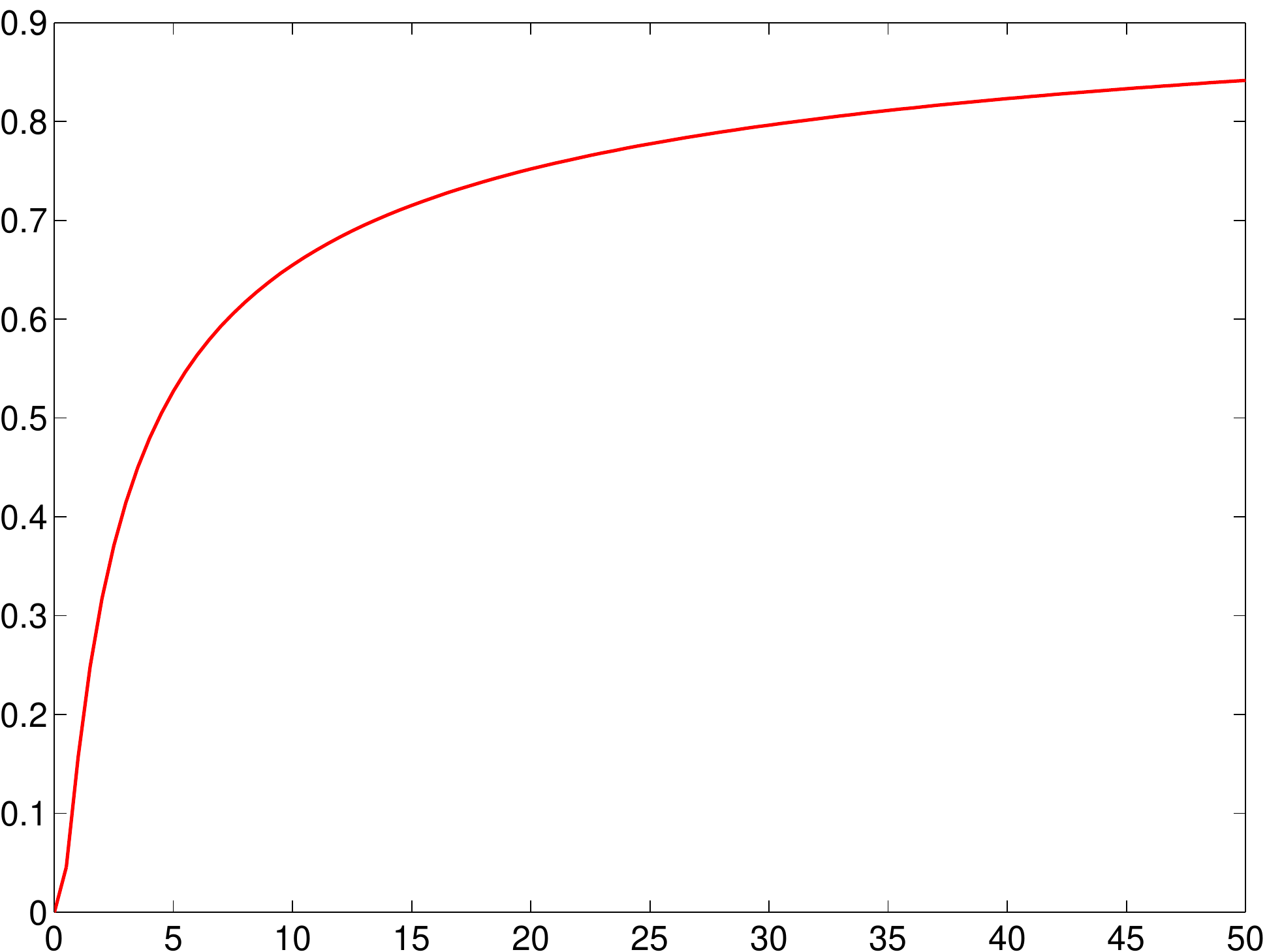}
      \caption{A plot of the cdf $t \mapsto \mathbb{P}(\tau_2 \leq t)$ for a Brownian motion $X$.}
      \label{Fig1}
    \end{center}
  \end{minipage}
  \hfill
  \begin{minipage}[t]{.45\textwidth}
    \begin{center}  
      \includegraphics[width=7.5cm]{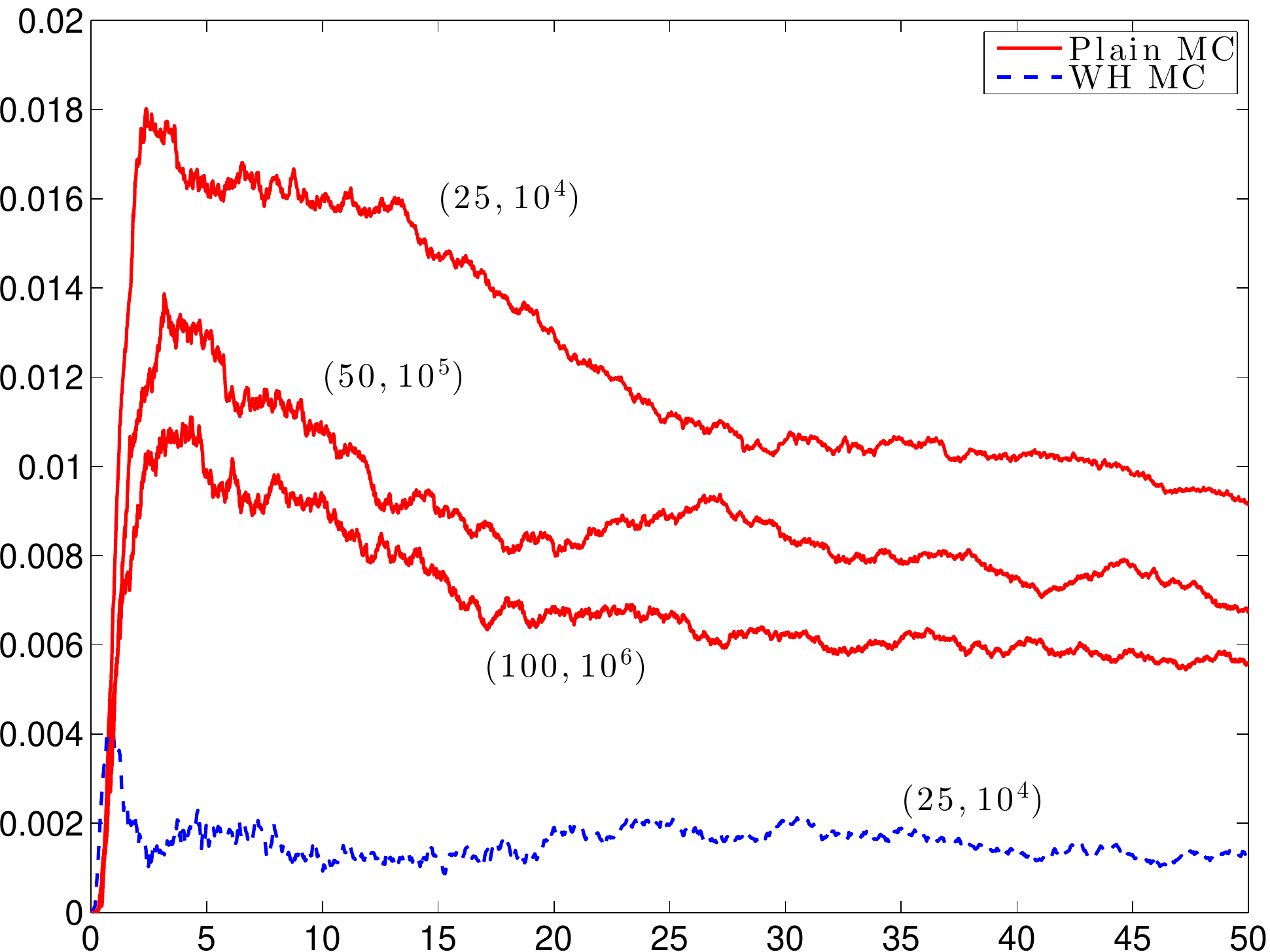}
      \caption{Absolute error in the approximation of the cdf of $\tau_2$ for plain Monte Carlo and the WHMC simulation technique. Coordinates $(n,m)$ stand for stepsize $1/n$ ($1/(2n)$ for plain Monte Carlo), and number of samples $m$.}
      \label{Fig2}
    \end{center}
  \end{minipage}
\end{figure}

\subsection{Convergence rates}\label{num_rates}

We now consider $X$ a driftless pure jump process belonging to the $\beta$-family described in Subsection \ref{beta_family}. In particular, we choose the coefficients in the L\'evy measure (\ref{beta_density})
to be $c_{i}=\beta_{i}=\lambda_{i}=1$ for $i=1,2$, $\alpha_{1}=1$ and $\alpha_{2}=2$. The chosen
values for $c_{i}$, $\beta_{i}$ and $\lambda_{i}$ make $X$ behave similar to the Variance Gamma process, a popular model in finance (see Schoutens and van Damme \cite{SD10} for the relation of the Variance Gamma process with the $\beta$-family). We choose different quantities for $\alpha_{1}$ and $\alpha_{2}$ to make the process asymmetric and ensure that a substantial quantity of the simulated paths will cross the barrier level which is set to $u=1$. We also set the monitoring time $t=1$. Variations of these parameters do not make significant changes in the following plots.

The plots in Figures \ref{Fig_rate_1} and \ref{Fig_rate_2} run a refinement sequence of Monte Carlo estimates for the mean square difference of consecutive levels of approximations for the first passage time, the overshoot, the undershoot and the last maximum before the first passage. In particular, \mbox{Figure \ref{Fig_rate_1}} depicts the expectation 
\begin{equation*}
\E\left[\left(\frac{t}{n_{\ell}} (\kappa^{(n_{\ell})}_u \wedge n_{\ell})-\frac{t}{n_{\ell-1}} (\kappa^{(n_{\ell-1})}_u \wedge n_{\ell-1})\right)^{2}\right]\ ,
\end{equation*}
for $n_{\ell}=2^{\ell}$ with $\ell=4,\ldots,10$. The plot uses 
$\log_{2}$-scales to show a clear decreasing rate of slope $1$. This convergence rate is dictated by the
result in Theorem \ref{conv_rate_tau} and the triangle inequality, hence  Figure \ref{Fig_rate_1} is a numerical evidence of Theorem \ref{conv_rate_tau}. In Figure \ref{Fig_rate_2} we proceed to repeat the same experiment for the overshoot, the undershoot and the last maximum before the first passage. We observe numerically that the decreasing rate is very close to $1/2$, i.e. 
\begin{equation*}
\E\left[\left( V_{\kappa^{(n_{\ell})}_u \wedge n_{\ell}}^{(n_{\ell}/t)}-V_{\kappa^{(n_{\ell-1})}_u \wedge n_{\ell-1}}^{(n_{\ell-1}/t)}\right)^{2}\right]\lesssim \frac{1}{n^{1/2}}\ ,
\end{equation*}
for the overshoot and analogous statements for undershoot and the last maximum before the passage. 
The above is a particular choice of a Cauchy sequence and therefore it can not prove
the convergence
\begin{equation*}
\left(V_{\kappa^{(n_{\ell})}_u \wedge n_{\ell}}^{(n_{\ell}/t)}-u\right)
\ \stackrel{{L}^{2}}{\longrightarrow}\ 
\left(X_{\tau_u \wedge t}-u\right)\ .
\end{equation*}
Nevertheless it suggests that the above convergence is plausible and indicates its possible convergence rate. The same conclusions can be derived for the overshoot and the last maximum before the first passage time according to Figure \ref{Fig_rate_2}.
\begin{figure}[t]
  \begin{minipage}[t]{.45\textwidth}
    \begin{center}  
      \includegraphics[width=7.5cm]{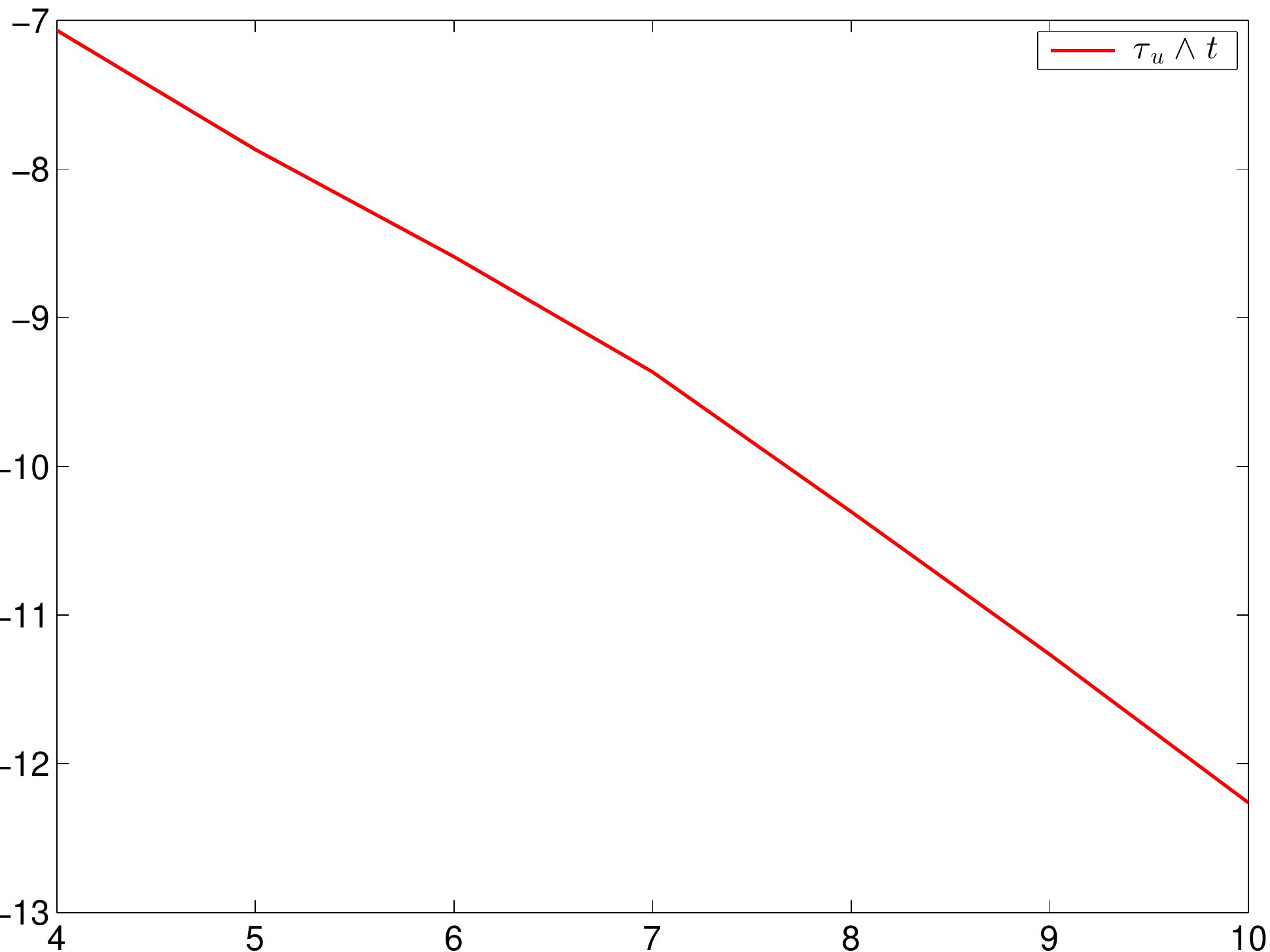}
      \caption{Mean square error for the difference in consecutive levels of the approximation for the first passage time (using $\log_{2}$-scales).}
      \label{Fig_rate_1}
    \end{center}
  \end{minipage}
  \hfill
  \begin{minipage}[t]{.45\textwidth}
    \begin{center}  
      \includegraphics[width=7.5cm]{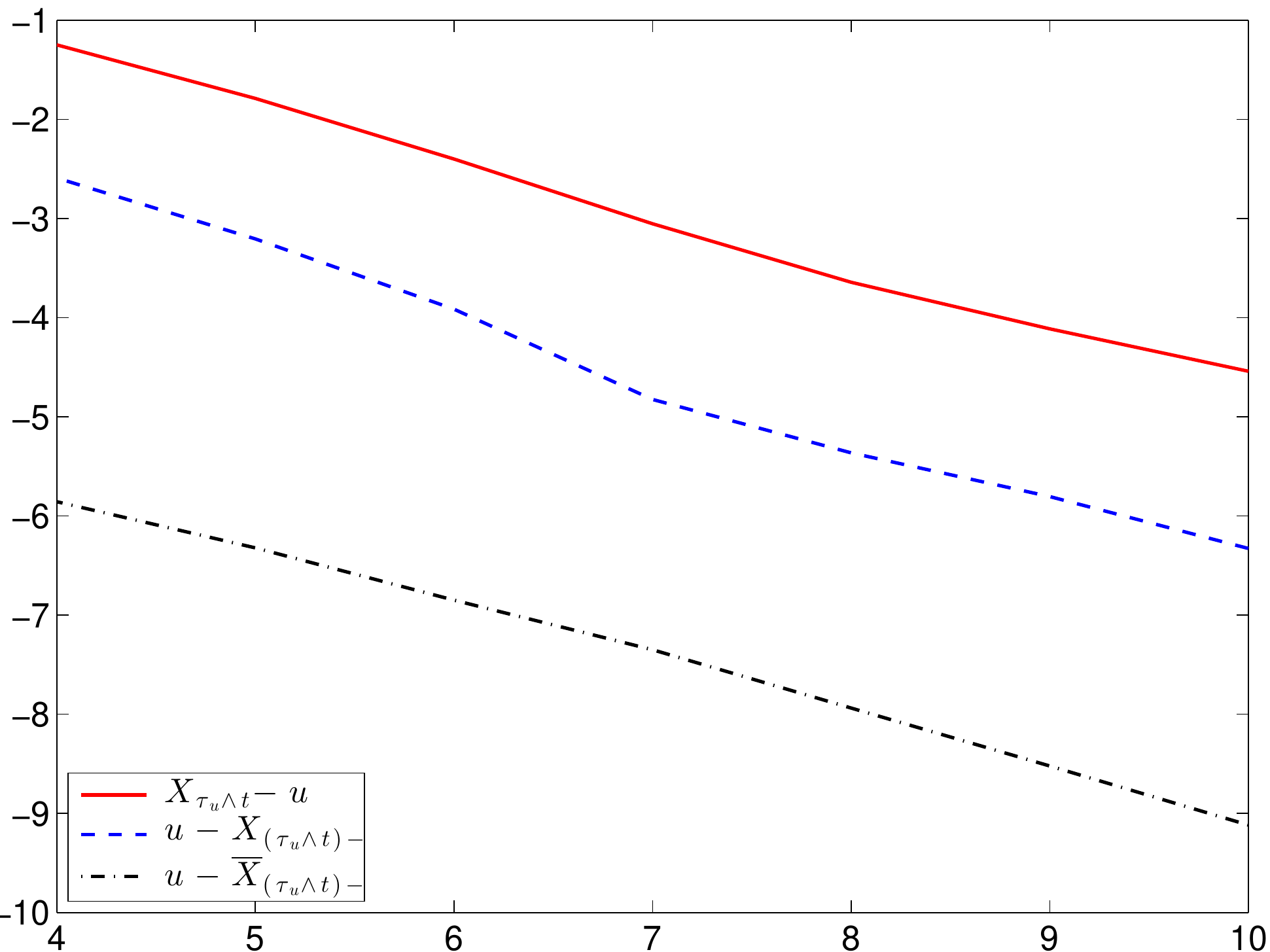}
      \caption{Mean square error for the difference in consecutive levels of the approximation for the overshoot, undershoot and last maximum before passage (using $\log_{2}$-scales).}
      \label{Fig_rate_2}
    \end{center}
  \end{minipage}
\end{figure}

\subsection{The joint law of first passage time and overshoot}\label{subsec_BetaJoint}

We conclude with a final example, namely consider 

\begin{equation}\label{K2}
v(u,y,q):=\E[e^{-q\tau_{u}}\mathbf{1}_{\{X_{\tau_{u}}-u\leq y\}}]\ .
\end{equation}
This expression depends both on the first passage time and the overshoot. Indeed it is a simplified version of the quintuple law mentioned in the Introduction, cf. (\ref{GS}). Note that the algorithm described in Section \ref{WH_scheme} will be the same for any functional $f$ in (\ref{main1}) and therefore we only provide here a simplified example of (\ref{main1}) for the sake of completeness. Note that any comprehensive numerical analysis of (\ref{main1}) will require some assumptions on the smoothness of $f$, see for instance Ferreiro-Castilla et al. \cite{FKSS12} for an error analysis of $\mathbb{E}[f(\overline{X}_{1},X_{1})]$ where $f$ is uniformly Lipschitz. Indeed, the above quantity is a simplified version of the Gerber-Shiu penalty function as used in the insurance literature (see e.g. \cite{Avram11} and the references therein).

For meromorphic L\'evy processes a closed form expression is available for $v$ --- which is a reason for choosing this example --- in terms of the roots and poles of the function $\zeta \mapsto q+\Psi(\mathrm{i}\zeta)$, cf. Theorem 3 in Kuznetsov \cite{Kuznetsov10}. We choose $X$ the same as in Subsection \ref{num_rates}. Figures \ref{Fig_F_1} and \ref{Fig_F_2} depict the performance of our algorithm as we decrease the exponential step rate $1/n$. We set the monitoring time $t=10$ to let the process cross the barrier in most of the samples. (Note that $\tau_{u}$ is defined as the first passage time when the process is left to run indefinitely, in a practical implementation such as this one we have to consider $\tau_{u}\wedge t$ rather. However due to the exponential decay, in (\ref{K2}) the effect of replacing $\tau_u$ by $\tau_{u}\wedge t$ can be neglected for $t$ large enough). Both plots in Figures \ref{Fig_F_1} and
\ref{Fig_F_2} show a similar behavior consistent with the results and convergence rates in \cite{FKSS12,FS11,SD10} --- which only approximated the joint distribution $(X_{t},\overline{X}_{t})$ --- and the results in Section \ref{num_rates}. This is a further indication that the approximation of the 4-tuple exhibits similar behaviour as the approximation of $(X_{t},\overline{X}_{t})$ derived in earlier work.

\begin{figure}[t]
  \begin{minipage}[t]{.45\textwidth}
    \begin{center}  
      \includegraphics[width=7.5cm]{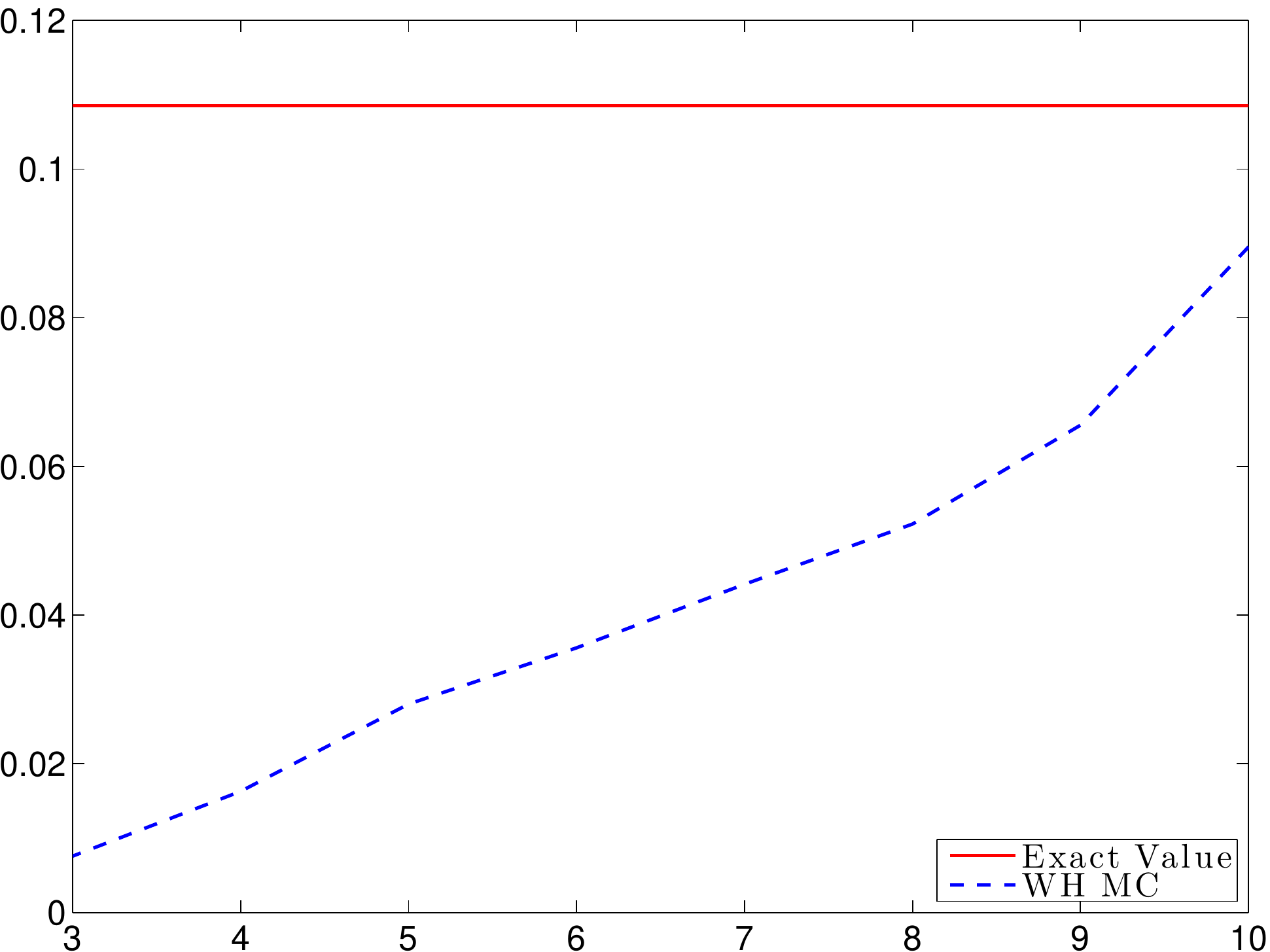}
      \caption{Approximation of $v(0.1,0.05,1)$ against rate of steps (in $\log_{2}$-scale) for the WHMC simulation method.}
      \label{Fig_F_1}
    \end{center}
  \end{minipage}
  \hfill
  \begin{minipage}[t]{.45\textwidth}
    \begin{center}  
      \includegraphics[width=7.5cm]{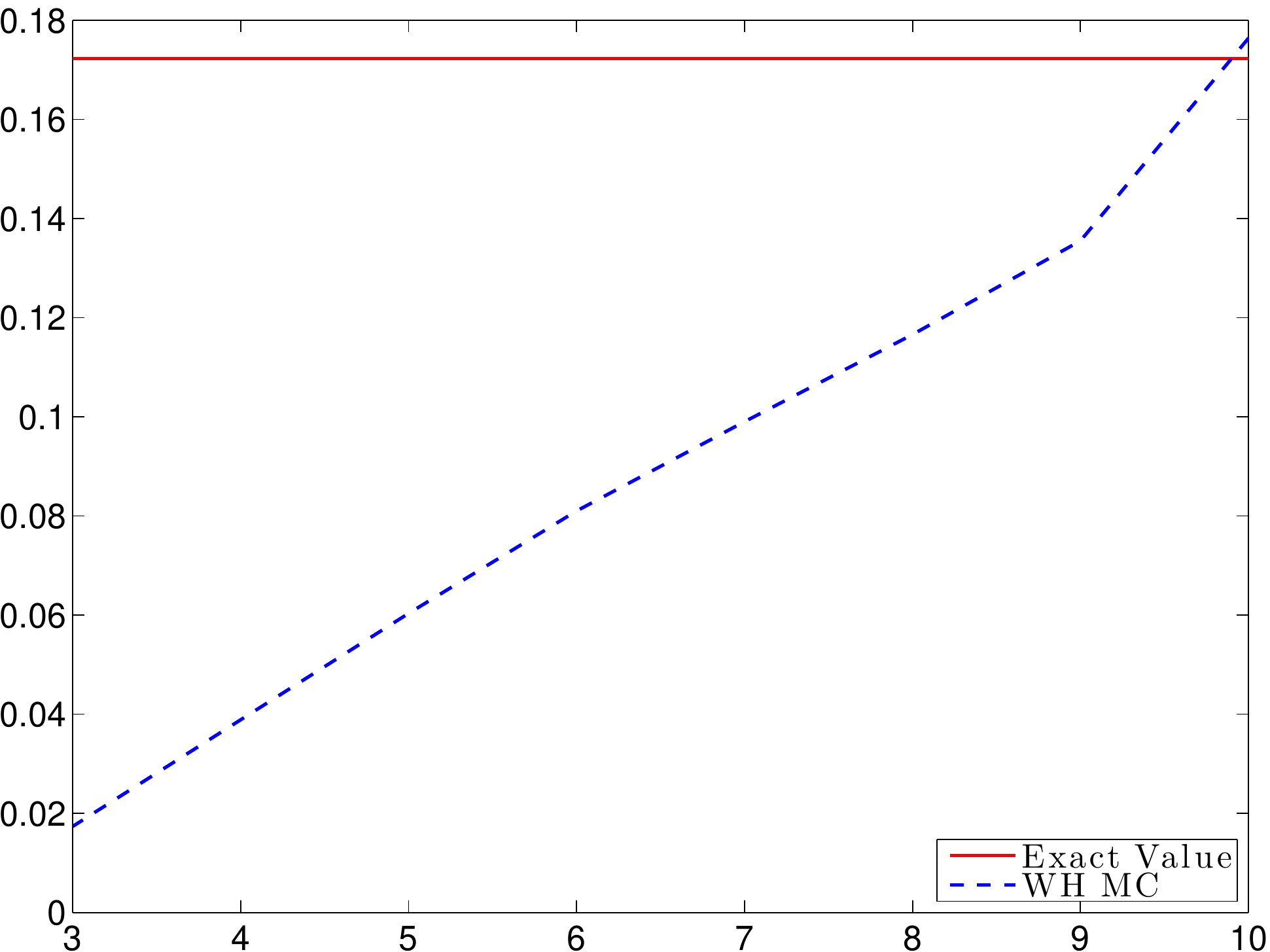}
      \caption{Approximation of $v(0.15,0.15,1)$ against rate of steps (in $\log_{2}$-scale) for the WHMC simulation method.}
      \label{Fig_F_2}
    \end{center}
  \end{minipage}
\end{figure}

\end{document}